\documentclass[twoside,11pt]{article}

\usepackage[english]{babel}
\usepackage{amsmath}
\usepackage{amsthm}
\usepackage{amsfonts}
\usepackage{amssymb}
\usepackage{graphicx}
\usepackage{colortbl,dcolumn}
\usepackage{paralist}  

       \newtheorem{definition}{\bf Definition}[section]
       \newtheorem{lemma}[definition]{\bf Lemma}
       \newtheorem{theorem}[definition]{\bf Theorem}

       \newtheorem{remark}[definition]{\bf Remark}

       \numberwithin{equation}{section}

\newcommand{\R}{\mathbb{R}}

\newcommand{\be}{\begin{equation}}
\newcommand{\ee}{\end{equation}}
\newcommand{\T}{\tau}
\newcommand{\rh}{\rho}
\newcommand{\bc}{\begin{equation*}}
\newcommand{\ec}{\end{equation*}}
\newcommand{\intr}{\int_{r_0}^{r_1}}
\newcommand{\rol}{[r_0,r_1]}
\newcommand{\ro}{(r_0,r_1)}
\newcommand{\hj}{\mathcal{J}}
\begin{document}

\title{{\LARGE Radial solutions of the hydrodynamic model of semiconductors with sonic boundary} \footnotetext{\small *Corresponding author.}
\footnotetext{\small E-mail addresses: chenl492@nenu.edu.cn (L. Chen), ming.mei@mcgill.ca (M. Mei), zhanggj100@nenu.edu.cn (G. Zhang), zhangkj201@nenu.edu.cn (K. Zhang)}}

\author{{}\\[2mm]
\small\it School of Mathematics and Statistics, Northeast Normal University,\\
\small\it   Changchun 130024, P.R.China \\}
\author{{Liang Chen$^{a}$, Ming Mei$^{b,c}$, Guojing Zhang$^{a,*}$ and Kaijun Zhang$^{a}$}\\[2mm]
\small\it $^a$School of Mathematics and Statistics, Northeast Normal University,\\
\small\it   Changchun 130024, P.R.China \\
\small\it $^b$Department of Mathematics, Champlain College Saint-Lambert,\\
\small\it     Saint-Lambert, Quebec, J4P 3P2, Canada\\
\small\it $^c$Department of Mathematics and Statistics, McGill University,\\
\small\it     Montreal, Quebec, H3A 2K6, Canada }

\date{}

\maketitle

\begin{quote}
\small \textbf{Abstract}:
The purpose of this paper is to study radial solutions for steady hydrodynamic model of semiconductors represented by Euler-Poisson equations with sonic boundary.
The existence and uniqueness of radial subsonic solution, and the existence of radial supersonic solutions are derived by using the energy method and the compactness method, but under a general condition of the doping profile. In particular, for radial supersonic solutions, it is more difficult  to get the related estimates by the effect of high dimensional space and the sonic boundary, so we apply a special iteration to complete the proofs. The results obtained essentially improve and develop the previous studies in the one-dimensional case.




\end{quote}


\section{Introduction}\label{Sect.1}

The hydrodynamic model of semiconductors, first introduced by Bl{\o}tekj{\ae}r \cite{Bl70}, usually characterizes the motion of the charged fluid particles such as electrons and holes in semiconductor devices \cite{Ma90}.
This paper is a follow-up  of our series of study \cite{Ch20,Li17,Li18} on the Euler-Poisson equations of semiconductor models subjected to sonic boundary. Different from \cite{Ch20,Li17,Li18} on 1-D equations with sonic boundary, here we are mainly interested in the multiple-dimensional Euler-Poisson system  \cite{Ba16,De93} as follows:
\begin{equation} \label {1.1}
\begin{cases}
\rho_t+\text{div}(\rho \mathbf{u})=0, \\
(\rho \mathbf{u})_t+\text{div}(\rho \mathbf{u}\otimes \mathbf{u}+PI_n)=\rho\nabla \Phi-\dfrac{\rho \mathbf{u}}{\tau},\\
\Delta\Phi=\rho-b(x).
\end{cases}
(x,t)\in \mathbb{R}^n\times\mathbb{R}^+,n=2,3,
\end{equation}
Here $\rho$, $\mathbf{u}$ and $\Phi$ denote the electron density, the velocity and the electrostatic potential, respectively. $I_n$ is the $n\times n$ identity matrix, the constant $\T>0 $ is the momentum relaxation time and the function $b(x)>0$ is the doping profile standing for the density of positively charged background ions. $P(\rho)$ is known as the pressure-density relation. As usual, for isentropic flows,
  $P(\rho) = \kappa\rho^{\gamma}$, $\kappa> 0$ with the adiabatic exponent $\gamma>1$; for isothermal flows, $P=T\rho$ with the constant temperature $T>0$. In present paper, we consider the isothermal case, and set $T=1$ without loss of generality, i.e.
\begin{equation*}
P(\rh)=\rh.
\end{equation*}

Throughout this paper, we consider the steady-state solutions of \eqref{1.1} in an annulus  domain
\[
\mathcal{A}:=\{x\in \mathbb{R}^n| r_0< |x|<r_1\},  \ \ \ 0<r_0<r_1,
\]
with the inner boundary
\begin{equation*}
\Gamma_0:=\{x\in\R^n: |x|=r_0\},
\end{equation*}
and the outer boundary
\[
\Gamma_1:=\{x\in\R^n:|x|=r_1\},
\]
Its closure is denoted by
\[
\overline{\mathcal{A}}:=\Gamma_0\cup\mathcal{A}\cup\Gamma_1.
\]
Note that
\begin{equation*}
\text{div}(\rho \mathbf{u}\otimes \mathbf{u})=\rho(\mathbf{u}\cdot\nabla)\mathbf{u}+\text{div}(\rho \mathbf{u})\cdot \mathbf{u},
\end{equation*}
and set $E:=\nabla \Phi$ (the electric field), then  the corresponding stationary equations of \eqref{1.1} can be written as
\begin{equation}\label{1.2}
\begin{cases}
\text{div}(\rho \mathbf{u})=0,\\
(\mathbf{u}\cdot\nabla)\mathbf{u}+\dfrac{\nabla \rho}{\rho}= E-\dfrac{\mathbf{u}}{\tau},\\
\text{div} ~E=\rho-b(x).
\end{cases}
x\in \mathcal{A},
\end{equation}

The aim of our work is to investigate the structure of the steady-state solutions to \eqref{1.2}, particularly, the radial subsonic/supersonic solutions of \eqref{1.2}  in two and three dimensional annulus domains with sonic boundary, and to study various analytical features including the requirement of the doping profile and the adopted methods in the proofs by comparing with the one-dimensional case \cite{Li17}.

Additionally, we call $M:=\frac{|\mathbf{u}|}{c(\rho)}$ the {\it Mach number} for $c(\rho) :=\sqrt{P'(\rho)}=1$. Here, $c(\rho)$ is called the {\it local sound speed}.  Depending on the size of $M$, the analytic features of \eqref{1.2} vary:  if $M>1$,  the stationary flow is called {\it supersonic}; if $M<1$, the corresponding flow is called {\it subsonic}; otherwise, $M=1$ is the sonic state.

In what follows, we assume that $\tilde{b}$ is in $L^{\infty}\ro$ such that $b(x):=\tilde{b}(r)$ in $\overline{\mathcal{A}}$, and we denote
\begin{equation}\label{a1.34}
\begin{split}
&(\rho,\mathbf{u},E)(x):=(\tilde{\rh}(r),\tilde{u}(r)\vec{\mathbf{e}}, \tilde{E}(r)\vec{\mathbf{e}}),
\end{split}
\end{equation}
where $r=|x|$, and  $\vec{\mathbf{e}}:=\dfrac{x}{r}$ is a unit vector,
and we prescribe the boundary conditions as follows:
\begin{equation}\label{a1.2}
(\rho|_{\Gamma_0}, \rho|_{\Gamma_1}, \rho \mathbf{u}|_{\Gamma_0})=(\tilde{\rho}(r_0),\tilde{\rho}(r_1), \tilde{\rho}(r_0)\tilde{u}(r_0)\vec{\mathbf{e}})=(\rho_0, \rho_1,j_0\vec{\mathbf{e}})
\end{equation}
for positive constants $(\rho_0, \rho_1 ,j_0)$. Therefore, \eqref{1.2} and \eqref{a1.2} is reduced to
\be\label{b1.5}
\begin{cases}
(r^{n-1}\tilde{\rho} \tilde{u})_r=0,\\
(r^{n-1}\tilde{\rho}\tilde{u}^2)_r+r^{n-1}\tilde{\rho}_r=r^{n-1}\tilde{\rho}(\tilde{E}-\frac{
\tilde{u}}{\T}),\\
(r^{n-1}\tilde{E})_r =r^{n-1}(\tilde{\rho}-\tilde{b}(r)),\\
(\tilde{\rh}(r_0),\tilde{\rh}(r_1),\tilde{u}(r_0))=(\rh_0,\rh_1, j_0/\rh_0),
\end{cases}
\quad \text{for}\quad r_0<r<r_1,
\ee
so that the sonic state is redefined by $|\tilde{u}|=M=1$.
Clearly, each pair of the solution $(\tilde{\rh},\tilde{u},\tilde{E})$ to system \eqref{b1.5} always corresponds to a solution  $(\rho, \mathbf{u},  E)$ to \eqref{1.2} and \eqref{a1.2}.

\begin{definition}[radial subsonic/supersonic solution]\label{de1.1}
We call $(\rho, \mathbf{u},  E)$ with $M<1$ $(M>1)$ in $\mathcal{A}$ radial subsonic (correspondingly, supersonic) to system \eqref{1.2} and \eqref{a1.2} if the corresponding solution $(\tilde{\rh},\tilde{u},\tilde{E})$ of \eqref{b1.5} satisfies $|\tilde{u}|<1$ $(|\tilde{u}|>1)$ over $\ro$.
\end{definition}

We now focus on \eqref{b1.5}. Let
$\tilde{J}:=\tilde{\rh}\tilde{u}$. Without loss of generality,  let us also take $\tilde{J}>0$.
From the first equation of \eqref{b1.5} we have
\be\label{a1.32}
\tilde{J}(r)=j_0\cdot\frac{r_0^{n-1}}{r^{n-1}},\qquad  r\in[r_0,r_1].
\ee
By \eqref{b1.5} and \eqref{a1.32}, we can impose the sonic boundary conditions to \eqref{b1.5} by
\begin{equation}\label{a1.33}
  \rho_0=j_0 \quad \text{and}\quad\rho_1=j_0\cdot\frac{r_0^{n-1}}{r_1^{n-1}}.
\end{equation}
By dividing the second equation of \eqref{b1.5} by $\tilde{\rh}$ and differentiating the resulting equation with respect to $r$, and using the third equation of \eqref{b1.5},  we obtain
\begin{equation}\label{1.6}
\begin{cases}
   \left[r^{n-1}\left(\left(\dfrac{1}{\tilde{\rho}}- \dfrac{\tilde{J}^2}{\tilde{\rho}^3}\right)\tilde{\rho}_r
  -\dfrac{n-1}{r}\dfrac{\tilde{J}^2}{\tilde{\rho}^2}+
  \dfrac{\tilde{J}}{\T\tilde{\rho}}\right)\right]_r=
   r^{n-1}(\tilde{\rho}-\tilde{b}), \qquad r\in\ro,\\
\rho_0=j_0,\quad\rho_1=j_0\cdot\frac{r_0^{n-1}}{r_1^{n-1}}.
\end{cases}
\end{equation}
In order to classify the radial solutions, it is convenient to introduce a new variable
\bc
m(r):=r^{n-1}\tilde{\rh}(r), \quad r\in\rol,
\ec
with a parameter $ \mathcal{J}:=j_0r_0^{n-1}>0$.
Thus, by \eqref{a1.32}, it implies that
\begin{equation}\label{b1.9}
\tilde{J}=\frac{\hj}{r^{n-1}}\quad \text{and}\quad \tilde{\rh}=\frac{m}{r^{n-1}},
\end{equation}
 then \eqref{1.6} is reduced to
\be\label{a1.4}
\begin{cases}
\left[r^{n-1}\left(\dfrac{1}{m}-\dfrac{\mathcal{J}^2}{m^3}\right)m_r+\dfrac{r^{n-1}\mathcal{J}}{\T m}\right]_r=m-B(r)+r^{n-3}(n-1)(n-2), \ \ \ r\in \ro,\\
m(r_0)=m(r_1)=\mathcal{J},
\end{cases}
\ee
where the function $B$ is defined by $B(r):=r^{n-1}\tilde{b}(r)$ on $\rol$.
Obviously,  $m>\mathcal{J}$ means that the flow is subsonic; correspondingly, $0<m<\mathcal{J}$ stands for the supersonic flow. Moreover, equation \eqref{a1.4} is elliptic but degenerate at the boundary, that causes us essential difficulties.

Now we define an interior subsonic/supersonic solution of \eqref{a1.4} in the weak sense, which is first introduced by \cite{Li17}.
\begin{definition}
$m(r)$ is called an interior subsonic $($correspondingly, interior supersonic$)$ solution of system \eqref{a1.4} if $m(r_0)=m(r_1)=\mathcal{J}$ and $m(r)>\mathcal{J}$ $($correspondingly, $0<m(r)<\mathcal{J})$ for $r\in(r_0,r_1)$, and $(m-\mathcal{J})^2\in H^1_0\ro$, and it holds that
$$
\intr \left[r^{n-1}\left(\frac{1}{m}-\frac{\mathcal{J}^2}{m^3}\right)m_r+\frac{r^{n-1}\mathcal{J}}{\T m}\right]\varphi_rdr+\intr (m-B(r)+r^{n-3}(n-1)(n-2))\varphi dr=0,
$$
for any $\varphi\in H_0^1(r_0,r_1)$, which is equivalent to
\begin{equation}\label{1.14}
\begin{split}
&\intr \left[r^{n-1}\frac{m+\mathcal{J}}{2m^3}((m-\hj)^2)_r+\frac{r^{n-1}\mathcal{J}}{\T m}\right]\varphi_rdr\\
&\quad+\intr (m-B(r)+r^{n-3}(n-1)(n-2))\varphi dr=0.
\end{split}
\end{equation}
\end{definition}
Once $m$ is known from \eqref{a1.4}, in view of \eqref{a1.34} and \eqref{b1.9}, ${\rh}$ and $\mathbf{{u}}$ can be determined.
Then, by the second equation of \eqref{1.2}, $\tilde{u}=\frac{\hj}{m}$ and \eqref{b1.9}, $E(x)$ is computed by
\begin{equation*}
  E(x)=\tilde{E}(r)\vec{\mathbf{e}}=\left(\tilde{u}\tilde{u}_r
  +\frac{\tilde{\rho_r}}{\tilde{\rho}}+\frac{\tilde{u}}{\T}\right)\vec{\mathbf{e}}=\left(\frac{(m+\mathcal{J})
  [(m-\mathcal{J})^2]_r}{2m^3}+\frac{\mathcal{J}}{\T m}-\frac{n-1}{r}\right)\vec{\mathbf{e}}.
\end{equation*}
Thus, finding the solution of \eqref{1.2}, \eqref{a1.2} and \eqref{a1.33} amounts to solving \eqref{a1.4}.
\begin{definition}\label{de1.3}
$(\rho, \mathbf{u},  E)$ is called a radial subsonic (correspondingly, supersonic) solution in $\mathcal{A}$ to \eqref{1.2} and \eqref{a1.2} with the sonic boundary conditions \eqref{a1.33} if the corresponding solution of \eqref{a1.4} is an interior subsonic (supersonic) solution.
\end{definition}

The study on the hydrodynamic system of semiconductors has been one of hot research spots  \cite{Ju01,Ma86,Ma91, Se89}.
  For the subsonic flows,  Degong and Markowich  \cite{De90,De93} first proved the existence and uniqueness of smooth solutions with a fully subsonic background in one dimension, and for potential flow in three dimensions, respectively; see \cite{Am01} for a non-isentropic case, and also \cite{Ba14,Ba16, Hu11, Hu2011, Li02, Su07,Su09} for more general subsonic case.  For the supersonic flows, the existence and uniqueness of supersonic solutions were studied by Peng and Violet \cite{Pe06} when the flow is strongly supersonic in the one-dimensional case. The work was extended to the two-dimensional case by Bae \cite{Du16}. Then, the transonic solutions have been a focus in the study of the stationary flows  because of the forming of shock waves, we refer to \cite{As91,Ba17,Du20, Ga92,Ga96,Lu11,Lu12, Ro05}.

For the one-dimensional case, if system \eqref{1.2} is with sonic boundary, the structures of all types of solutions for \eqref{1.2} have been intensively studied  when the doping profile is subsonic \cite{Li17}, supersonic \cite{Li18} or transonic \cite{Ch20}.
In the case of the subsonic \cite{Li17} and subsonic-dominated \cite{Ch20} doping profile, there exist a unique interior subsonic, at least one interior supersonic solution, infinitely many transonic shock solutions (the sufficiently large relaxation time, i.e. $\tau\gg1)$, and infinitely many $C^1$-smooth transonic solutions $($the sufficiently small relaxation time, i.e. $\tau\ll1)$. The approach adopted consists of the technical compactness analysis, phase plane analysis and the energy method. Of course, interior subsonic/supersonic solutions may not exist with the subsonic-dominated \cite{Ch20} doping profile if the relaxation time is small enough.  On the other hand, under the supersonic \cite{Li18} and supersonic-dominated \cite{Ch20} doping profile, the non-existence of all types of the solutions can be obtained. However, the existence of supersonic and transonic shock solutions can be proved in an extreme case, where  the doping profile is close to the sonic line and the semiconductor effect is small $(\tau\gg1)$.

Inspired by our previous studies mentioned above, we expect to establish the well-posedness of the solutions for high dimensional system with sonic boundary. Physically speaking, it is hard to put forward an acceptable critical boundary in a general domain, such as a flat nozzle.
 Therefore, we first pay attention to radial solutions of \eqref{1.2} in an annulus domain.
The work of this paper is to show that, given constant date $(\rh_0,j_0)$ at the inner boundary $\Gamma_0$ and constant density $\rh_1$ at the outer boundary $\Gamma_1$,  there exist a unique radial subsonic solution and at least one radial supersonic solution to \eqref{1.2} and \eqref{a1.2} with the sonic boundary conditions \eqref{a1.33}.

Since it is complicated to solve \eqref{1.2} directly, by Definition \ref{de1.3}, we turn to consider the interior subsonic/supersonic solutions of \eqref{a1.4} in the bounded interval.
Unsurprisingly, there still exist a unique interior subsonic solution and at least one interior supersonic solution to \eqref{a1.4}.
Afterwards, there are two different features from the one-dimensional case: the first key is that the requirement of the doping profile actually become more general, namely, the lower bound of the doping profile may be smaller than the sonic curve; the second finding is that a two-steps iteration, replacing the one-step iteration, is established to prove the existence of interior supersonic solutions for \eqref{a1.4}.

Throughout this paper we denote
 \begin{equation*}
 \underline{B}=\mathop{\mathrm{essinf}}\limits_{r\in\rol} B(r)\quad\text{and}\quad
   \overline{B}:=\mathop{\mathrm{esssup}}\limits_{r\in\rol} B(r),
 \end{equation*}
and also define
\begin{equation*}
\underline{\mathcal{B}}:=\inf\limits_{r\in\rol}\left\{B(r)+\dfrac{2r}{\T}-2\right\}
  \quad\text{and}\quad \overline{\mathcal{B}}:=\sup\limits_{r\in\rol}\left\{B(r)+\dfrac{2r}{\T}-2\right\},
\end{equation*}
 which is necessary to prove the existence of the solutions in the three-dimensional case.

Now we state our main results about interior subsonic/supersonic solutions to \eqref{a1.4} as follows.
\begin{theorem}[Interior subsonic solutions] \label{t1.1}
\begin{enumerate}
\item The case of $n=2$: Let $B(r)\in L^{\infty}(r_0,r_1)$ and $\underline{B}\leq B(r)\leq \overline{B}$ satisfying $\overline{B}+\dfrac{1}{\T}> \mathcal{J}$ and $\underline{B}+\dfrac{\mathcal{J}}{\T(\overline{B}+1/\T)}> \mathcal{J}$, then system \eqref{a1.4} admits a unique interior subsonic solution $m(r)$ over $\rol$. Further, $m\in C^{\frac{1}{2}}\rol$ satisfies a lower bound estimate
\begin{equation*}
  m(r)\geq \mathcal{J}+\lambda\sin \left(\pi\cdot\frac{r-r_0}{r_1-r_0}\right),\quad r\in\rol,
\end{equation*}
where $\lambda$ is a small and positive constant.
\item The case of $n=3$:  Let $\overline{\mathcal{B}}> \mathcal{J}$ and $\min\limits_{r\in\rol}\left(B(r)+\dfrac{2 r \mathcal{J}}{\T \overline{\mathcal{B}}}-2\right)> \mathcal{J}$, then equation \eqref{a1.4} has a unique interior subsonic solution $m$ satisfying $m\in C^{\frac{1}{2}}\rol$ and
\begin{equation*}
  m(r)\geq \mathcal{J}+\bar{\lambda}\sin \left(\pi\cdot\frac{r-r_0}{r_1-r_0}\right),\quad r\in\rol,
\end{equation*}
where the constant $\bar{\lambda}>0$ is also small.
\end{enumerate}
\end{theorem}

\begin{theorem}[Interior supersonic solutions]\label{t1.2}
\begin{enumerate}
\item The case of $n=2$: Assume that $\underline{B}+\dfrac{1}{\T}> \mathcal{J}$, then system \eqref{a1.4} has at least one interior supersonic solution $m\in C^{1/2}\rol$ satisfying $\ell\leq m(r)< \mathcal{J}$ over $\ro$ for a positive constant $\ell$.
\item The case of $n=3$:
Suppose that $\underline{\mathcal{B}}> \mathcal{J}$, then there exists an interior supersonic solution $m\in C^{1/2}\rol$ to system \eqref{a1.4} satisfying $\bar{\ell}\leq m(r)< \mathcal{J}$ over $\ro$ for a positive constant $\bar{\ell}$.
\end{enumerate}
\end{theorem}
\begin{remark}{~}
\begin{enumerate}
\item If the hypotheses of Theorem \ref{t1.1} hold, we notice that subsonic solutions and  supersonic solutions of \eqref{a1.4} both exist. In addition, a higher requirement of the doping profile can be needed in three dimensional space.
\item For any fixed $r_0>0$, there exist always an interior subsonic solution and an interior supersonic solution to \eqref{a1.4} when the hypotheses of Theorem \ref{t1.1} and \ref{t1.2} are satisfied.
\item  Affected by high dimensions space, \eqref{1.2} will be recast as a nonlinear non-autonomous ODE system, which is more complex than autonomous system in one dimensional case. Thus, the transonic solutions of \eqref{a1.4} are not discussed in this paper, which will be left in the future.
\end{enumerate}
\end{remark}

Next the rest of this paper is organized as follows. The second section focuses on interior subsonic solutions of system \eqref{a1.4}. For clarity, we discuss this issue in the two-dimensional and three-dimensional cases, respectively.  Under both two cases, there exists a unique interior subsonic solution to \eqref{a1.4}. In addition, the third section  is devoted to interior supersonic solutions of \eqref{a1.4} in two and three dimensions cases. The existence of interior supersonic solutions is proved by a two-steps iteration and the Schauder fixed point theorem.

\section{Existence and uniqueness of interior subsonic solutions}
In this section, we're going to prove that there exists a unique interior subsonic solution to \eqref{a1.4} for both two-dimensional and three-dimensional cases. Here the main approach is the technical compactness method \cite{Li17}, which is inspired by the vanishing viscosity method.

{\it 2.1. the case of n=2. }  First we will prove the well-posedness of system \eqref{a1.4} in the two-dimensional case. Actually, we consider the following equation,
\be\label{2.1}
\begin{cases}
\left[r\left(\dfrac{1}{m}-\dfrac{\mathcal{J}^2}{m^3}\right)m_r+\dfrac{r \mathcal{J}}{\T m}\right]_r=m-B(r),\quad r\in\ro,\\
m(r_0)=m(r_1)=\mathcal{J}.
\end{cases}
\ee
Our main theorem in this subsection is stated below.

\begin{theorem}\label{t2.1}
Assume that $B(r)\in L^{\infty}(r_0,r_1)$ and $\underline{B}\leq B(r)\leq \overline{B}$ satisfying $\overline{B}+\dfrac{1}{\T}> \mathcal{J}$ and $\underline{B}+\dfrac{\mathcal{J}}{\T(\overline{B}+1/\T)}>\mathcal{J}$, then we have a unique weak solution $m$ to \eqref{2.1} satisfying $m\in C^{\frac{1}{2}}\rol$ and
\begin{equation*}
  m(r)\geq \mathcal{J}+\lambda\sin \left(\pi\cdot\frac{r-r_0}{r_1-r_0}\right),\quad r\in\rol,
\end{equation*}
where $\lambda$ is a small and positive constant.
\end{theorem}

Since \eqref{2.1} is elliptic in $\ro$ but degenerates at the boundary, we can't directly work on it. Therefore, we first consider the approximate equation of \eqref{2.1} as follows:
\be\label{a1.5}
\begin{cases}
\left[r\left(\dfrac{1}{m_j}-\dfrac{j^2}{(m_j)^3}\right)(m_j)_r+\dfrac{r \hj}{\T m_j}\right]_r=m_j-B(r),\quad r\in \ro,\\
m_j(r_0)=m_j(r_1)=\mathcal{J},
\end{cases}
\ee
where the parameter $j$ is a constant such that $0<j<\mathcal{J}$. Obviously, one finds that \eqref{a1.5} is uniformly elliptic in $\rol$ for the expected solution $m_j>\mathcal{J}$.
The following  comparison principle is the key ingredient to prove the uniqueness of interior subsonic solution to \eqref{a1.4}.
\begin{lemma}\label{l2.1}
Let $p\in C^1[r_0,r_1]$ be a weak solution of \eqref{a1.5} satisfying $p\geq \mathcal{J}$ on $[r_0,r_1]$, and
\begin{equation*}
\intr \left[r\left(\dfrac{1}{p}-\dfrac{j^2}{p^3}\right)p_r+\dfrac{r \hj}{\T p}\right]\varphi_rdr+\intr (p-B(r))\varphi dr=0
\end{equation*}
for any $\varphi\in H^1_0(r_0,r_1)$ where $0<j<\hj$. Further, let $q\in C^1[0,1]$ be such that $q(x)> 0$ on $[r_0,r_1]$, $q(r_0)\leq \mathcal{J}$, $q(r_1)\leq \mathcal{J}$, and for any $\varphi\geq0$,  $\varphi\in H^1_0(r_0,r_1)$,
\begin{equation*}
\intr \left[r\left(\dfrac{1}{q}-\dfrac{j^2}{q^3}\right)q_r+\dfrac{r \hj}{\T q}\right]\varphi_rdr+\intr (q-B(r))\varphi dr\leq0.
\end{equation*}
Then $p(r)\geq q(r)$ over $[r_0,r_1]$.
\end{lemma}
\begin{proof}
This proof is same as that of Lemma 2.2 \cite{Li17} and we omit it here.
\end{proof}
Now let's prove the well-posedness of \eqref{a1.5} first.
\begin{lemma}\label{l2.3}
Assume that $B(r)\in L^{\infty}(r_0,r_1)$ and $\underline{B}\leq B(r)\leq \overline{B}$ satisfying $\overline{B}+\dfrac{1}{\T}> \mathcal{J}$ and $\underline{B}+\dfrac{\mathcal{J}}{\T(\overline{B}+1/\T)}>\mathcal{J}$, then there exists  a unique weak solution $m_j$ to \eqref{a1.5} satisfying $m_j-\hj\in H^1_0\ro$ and
\begin{equation*}
  m_j(r)\geq \mathcal{J}+\lambda\sin \left(\pi\cdot\frac{r-r_0}{r_1-r_0}\right),\quad r\in\rol,
\end{equation*}
where $\lambda$ is a small and positive constant, independent of $j$.
\end{lemma}

\begin{proof}

First denote a closed subset of $C^0\rol$ by
\bc
\mathcal{C}:=\left\{\omega\in C^0[r_0,r_1]|\hj\leq\omega(r)\leq N,\omega(r_0)=\omega(r_1)=\mathcal{J}\right\}
\ec
 for a undetermined constant $N>\hj$. Then we define a fixed-point operator $\mathfrak{P}:\mathcal{C}\longrightarrow C^0\rol$, $\mathfrak{P}(\bar{m})=m_j$, by solving the linearized system of \eqref{a1.5},
\begin{equation}\label{a1.13}
\begin{cases}
\left[r\left(\dfrac{1}{\bar{m}}-\dfrac{j^2}{\bar{m}^3}\right)(m_j)_r\right]_r-\dfrac{r \hj}{\T \bar{m}^2}(m_j)_r=m_j-B(r)-\dfrac{\hj}{\T \bar{m}},\quad r\in \ro,\\
m_j(r_0)=m_j(r_1)=\mathcal{J},
\end{cases}
\end{equation}
with $\bar{m}\in\mathcal{C}$. Due to the $L^2$ theory of elliptic equations, we have $m_j\in H^1(r_0,r_1)$ for system \eqref{a1.13}. By the compact imbedding $H^1\ro\hookrightarrow C^0\rol$, one can see that $\mathfrak{P}(\bar{m})$ is precompact. Further, $\mathfrak{P}$ is continuous by a standard continuity argument. In order to use the Schauder fixed point theorem \cite{Di01}, it remains to prove $\mathfrak{P}(\mathcal{C})\subset \mathcal{C}$.

Here we only need to show $\mathcal{J}\leq m_j(r)\leq N$ over $\rol$ by selecting a suitable $N$. In fact,
if $B(r)+ \dfrac{\hj}{\T \bar{m}}\geq \mathcal{J}$ over $\rol$, we obtain
\begin{equation*}
\begin{cases}
 \left[r\left(\dfrac{1}{\bar{m}}-\dfrac{j^2}{\bar{m}^3}\right)(m_j)_r\right]_r-\dfrac{r \hj}{\T \bar{m}^2}(m_j)_r-(m_j-\hj)\leq 0,\quad r\in \ro,\\
 m_j(r_0)=m_j(r_1)=\mathcal{J}.
\end{cases}
\end{equation*}
Thus, by the weak maximum principle (Theorem 8.1 \cite{Di01}), it is easy to see that $m_j-\hj\geq 0$. Similarly, suppose that $B(r)+ \dfrac{\hj}{\T \bar{m}}\leq N$ over $\rol$, then it follows that
\begin{equation*}
\begin{cases}
 \left[r\left(\dfrac{1}{\bar{m}}-\dfrac{j^2}{\bar{m}^3}\right)(m_j)_r\right]_r-\dfrac{r \hj}{\T \bar{m}^2}(m_j)_r-(m_j-N)\geq 0,\quad r\in \ro,\\
 m_j(r_0)=m_j(r_1)=\mathcal{J},
\end{cases}
\end{equation*}
which yields that $m_j-N\leq 0$.
In brief, we can derive that $\mathcal{J}\leq m_j(r)\leq N$ over $\rol$ while
\begin{equation}\label{a2.5}
 \mathcal{J}\leq B(r)+\dfrac{\hj}{\T \bar{m}} \leq N,\quad r\in\rol,
\end{equation}
for arbitrary $\mathcal{J}\leq\bar{m}\leq N$.  Now we choose $N= \overline{B}+\dfrac{1}{\T }> \mathcal{J}$ so that the right-side inequality of \eqref{a2.5} directly holds. Moreover, a simple computation using the condition
\bc
\underline{B}+\dfrac{\hj}{\T(\overline{B}+1/\T)} > \mathcal{J}
\ec
yields that the left-side inequality of \eqref{a2.5} also holds.
Therefore, $\mathfrak{P}(\mathcal{C})\subset \mathcal{C}$, and one can see that there exists a fixed point $m_j$ of $\mathfrak{P}$ such that $\mathfrak{P}(m_j)=m_j$.
Recalled Theorem 1 of \cite{De90}, \eqref{a1.5} has a weak solution $m_j\in H^2(r_0,r_1)$.
Thanks to the compact imbedding $H^2(r_0,r_1)\hookrightarrow C^1\rol$, we have $m_j\in C^1\rol$.

Then we need to prove the uniqueness of the solution of \eqref{a1.5} and build a lower bound estimate. Suppose that there exist two solutions $m_j^1$ and $m_j^2$ satisfying $m_j^1$, $m_j^2\geq \mathcal{J}$ and $m_j^1$, $m_j^2\in C^1\rol$. Thus, Lemma \ref{l2.1} implies that $m_j^1(r)=m_j^2(r)$ over $\rol$. Furthermore, define
\begin{equation*}
\mathfrak{m}(r):=\mathcal{J}+\lambda \sin \left(\pi\cdot\frac{r-r_0}{r_1-r_0}\right),\quad r\in\rol,
\end{equation*}
where $\lambda$ is a positive constant.  Note that $\underline{B}+\dfrac{\hj}{\T(\overline{B}+1/\T)}>\mathcal{J}$, then a direct calculation shows that
\begin{equation*}
\begin{split}
  -\left[r\left(\dfrac{1}{\mathfrak{m}}-\dfrac{j^2}{\mathfrak{m}^3}\right)\mathfrak{m}_r+\dfrac{r \hj}{\T \mathfrak{m}}\right]_r+\mathfrak{m}-B(r)& \leq C(\lambda^2+\lambda)+\left(\mathcal{J}-B(r)-\frac{\hj}{\T(\mathcal{J}+\lambda)}\right)\\
  &< C(\lambda^2+\lambda)+\left(\mathcal{J}-B(r)-\frac{\hj}{\T(\overline{B}+1/\T)}\right)\\
  &<0,
\end{split}
\end{equation*}
by choosing $\lambda$ sufficiently small to satisfy $\lambda<\overline{B}+\dfrac{1}{\T}-\mathcal{J}$ and
$C(\lambda^2+\lambda)<\underline{B}+\dfrac{\hj}{\T(\overline{B}+1/\T)}-\mathcal{J}$.
Here $C=C(\T,r_0)$ is a positive constant independent of $j$. Hence, by Lemma \ref{l2.1}, we get that
\begin{equation}\label{m2.7}
m_j(r)\geq \mathfrak{m}(r)\quad \text{over}\quad \rol,
\end{equation}
and the constant $\lambda$ is positive and small, independent of $j$. The proof is complete.
\end{proof}
Next we return to prove Theorem \ref{t2.1}.
\begin{proof}[Proof of Theorem \ref{t2.1}]
Multiplying \eqref{a1.5} by $(m_j-\mathcal{J})$, we get
\be\label{2.5}
\begin{split}
   & (\mathcal{J}^2-j^2)\intr r\dfrac{|(m_j)_r|^2}{(m_j)^3}dr+\frac{4}{9}\intr r\frac{m_j+\mathcal{J}}{(m_j)^3}|[(m_j-\mathcal{J})^{\frac{3}{2}}]_r|^2 dr\\
   &+\frac{\hj}{\T}\intr \frac{r(m_j)_r}{m_j}dr+\intr (m_j-B)(m_j-\mathcal{J})dr=0.
\end{split}
\ee
 Combining $\underline{B}+\dfrac{1}{\T}\geq \mathcal{J}$ with integration by parts and Cauchy inequality, we obtain
\begin{equation*}
  \frac{\hj}{\T}\intr \frac{r(m_j)_r}{m_j}dr = \frac{\hj}{\T}\intr rd(\ln m_j)=\frac{\hj}{\T}[(r_1-r_0)\cdot\ln \mathcal{J}]-\frac{\hj}{\T}\intr \ln m_jdr,
\end{equation*}
and
\begin{equation*}
\begin{split}
 \intr (m_j-B)(m_j-\mathcal{J})dr
&  \geq \intr (m_j-\mathcal{J})^2dr-\intr \left(B+\frac{1}{\T}-\mathcal{J}\right)(m_j-\mathcal{J})dr\\
&\quad+\frac{1}{\T}\intr (m_j-\mathcal{J})dr\\
 &\geq \frac{1}{2}\intr(m_j-\mathcal{J})^2dr-\frac{1}{2}\intr \left(B+\frac{1}{\T}-\mathcal{J}\right)^2dr\\
 &\quad+\frac{1}{\T}\intr (m_j-\mathcal{J})dr.
 \end{split}
\end{equation*}
After that, because of $\mathcal{J}\leq m_j\leq \overline{B}+\dfrac{1}{\T}$, we derive from \eqref{2.5} that
\begin{equation*}
\begin{split}
&\frac{(\mathcal{J}^2-j^2)r_0}{(\overline{B}+\frac{1}{\T})^3}\intr  |(m_j)_r|^2dr+\frac{8 r_0\mathcal{J}}{9(\overline{B}+\frac{1}{\T})^3}\intr |[(m_j
-\mathcal{J})^{\frac{3}{2}}]_r|^2dr+\frac{1}{2}\intr(m_j-\mathcal{J})^2dr\\
&\quad\leq\frac{1}{2}\intr \left(B+\frac{1}{\T}-\mathcal{J}\right)^2dr+\frac{\hj}{\T}\intr \ln m_jdr-\frac{\hj}{\T}[(r_1-r_0)\cdot\ln \mathcal{J}]\\
&\quad\leq\frac{1}{2}\intr \left(\overline{B}+\frac{1}{\T}-\mathcal{J}\right)^2dr+\frac{\mathcal{J}(r_1-r_0)}{\T}\left[\ln \left(\overline{B}+\frac{1}{\T}\right)-\ln \mathcal{J}\right],
\end{split}
\end{equation*}
which gives
\begin{equation*}
  \|(m_j-\mathcal{J})^{\frac{3}{2}}\|_{H^1\ro}\leq C_1(\overline{B},\T,r_0),\quad \|(\mathcal{J}^2-j^2)(m_j)_r\|_{L^2\ro}\leq C_2(\overline{B},\T,r_0)(\mathcal{J}^2-j^2)^{\frac{1}{2}}.
\end{equation*}
Here $C_1$ and  $C_2$ are positive constants independent of $j$.
Thus, by the compact imbedding $H^1\ro\hookrightarrow C^{\alpha}\rol$, $0<\alpha<\frac{1}{2}$, there exists a function $m$, as $j\rightarrow \mathcal{J}^-$,  such that up to a subsequence,
\begin{equation}\label{b2.10}
(m_j-\mathcal{J})^{\frac{3}{2}}\rightharpoonup
(m-\mathcal{J})^{\frac{3}{2}}\quad\text{weakly ~in} \quad H^1(r_0,r_1),
\end{equation}
\begin{equation}\label{b2.11}
 (m_j-\mathcal{J})^{\frac{3}{2}}\rightarrow
 (m-\mathcal{J})^{\frac{3}{2}}\quad\text{strongly ~in} \quad C^\alpha[r_0,r_1],
\end{equation}
\begin{equation}\label{b2.12}
  (\mathcal{J}^2-j^2)(m_j)_r\rightarrow 0 \quad\text{strongly ~in} \quad L^2(r_0,r_1).
\end{equation}
Noticing that $[(m_j-\mathcal{J})^2]_r=\frac{4}{3}(m_j-\mathcal{J})^{\frac{1}{2}}[(m_j-\mathcal{J})^{\frac{3}{2}}]_r$, we get
\begin{equation*}
  ||(m_j-\mathcal{J})^2||_{H^1\ro}\leq C||(m_j-\mathcal{J})^{\frac{3}{2}}||_{H^1\ro}\leq C(r_0,\overline{B},\T),
\end{equation*}
which leads to
\begin{equation*}
  (m_j-\mathcal{J})^2\rightharpoonup (m-\mathcal{J})^2\quad\text{weakly ~in} \quad H^1(r_0,r_1) \quad\text{as}\quad j\rightarrow \mathcal{J}^-.
\end{equation*}
Thus, multiplying \eqref{a1.5} by $\varphi\in H_0^1(r_0,r_1)$, we have
\begin{equation*}
\begin{split}
&\intr r\frac{m_j+\mathcal{J}}{2m_j^3}[(m_j-\hj)^2]_r\varphi_rdr+\intr \frac{r}{m_j^3}(\hj^2-j^2)(m_j)_r\varphi_rdr\\
&\quad+\dfrac{\hj}{\T}\intr \frac{r}{ m_j}\varphi_rdr+\intr (m-B(r))\varphi dr=0.
\end{split}
\end{equation*}
As $j\rightarrow \mathcal{J}^-$, by \eqref{b2.10}-\eqref{b2.12}, \eqref{1.14} holds in the case of $n=2$.
The lower bound estimate is directly obtained from \eqref{m2.7} and \eqref{b2.11}.

To prove the uniqueness of the interior subsonic solution, we first define $w(r):=(m(r)-\mathcal{J})^2$ and it is easy to see that $w\in H_0^1(r_0,r_1)$ satisfies the equality
\be\label{2.6}
\left(\dfrac{r(\sqrt{w}+2\mathcal{J})w_r}{2(\sqrt{w}+\mathcal{J})^3}+\dfrac{r \mathcal{J}}{\T (\sqrt{w}+\hj)}\right)_r=\sqrt{w}+\mathcal{J}-B(r),\quad r\in\ro.
\ee
Then, recalled from the proof of Theorem 2.1 \cite{Li17}, it implies by \eqref{2.6} that $w\in C^{1+\frac{1}{4}}[r_0,r_1]$.
Letting
\begin{equation*}
G_w(r):=\dfrac{r(\sqrt{w}+2\mathcal{J})w_r}{2(\sqrt{w}+\mathcal{J})^3}+\dfrac{r\mathcal{J}}{\T (\sqrt{w}+\mathcal{J})},
\end{equation*}
we have
\be\label{2.7}
\begin{cases}
\dfrac{r(\sqrt{w}+2\mathcal{J})w_r}{2(\sqrt{w}+\mathcal{J})^3}=G_w-\dfrac{r \mathcal{J}}{\T (\sqrt{w}+\mathcal{J})},\\
G_w(r)=G_w(r_0)+\int_{r_0}^r (\sqrt{w(s)}+\mathcal{J}-B(r))ds.
\end{cases}
\ee
First, suppose that \eqref{2.1} has two different interior subsonic solutions $m_1(r)$ and $m_2(r)$ over $\rol$. Next there exists a nonempty domain $[\bar{r}_0,\bar{r}_1]\subset [r_0.r_1]$ such that \eqref{2.6} has two corresponding solutions $w_1(r)$ and $w_2(r)$ satisfying
\begin{equation*}
w_1(\bar{r}_0)=w_2(\bar{r}_0),\quad w_1(\bar{r}_1)=w_2(\bar{r}_1)\quad \text{and} \quad w_1(r)> w_2(r)\quad \text{for}\quad r\in(\bar{r}_0,\bar{r}_1).
\end{equation*}
Because of the $C^1$-continuity of $w_1$ and $w_2$, it holds that
\begin{equation}\label{a2.11}
(w_1)_r(\bar{r}_0)\geq (w_2)_r(\bar{r}_0)\quad \text{and} \quad (w_1)_r(\bar{r}_1)\leq (w_2)_r(\bar{r}_1).
\end{equation}
Hence, it follows from the first equation of $\eqref{2.7}$ that $G_{w_1}(\bar{r}_1)\leq G_{w_2}(\bar{r}_1)$. Then by the second equation of \eqref{2.7}, we derive
\begin{equation*}
  G_{w_1}(\bar{r}_0)+\int_{\bar{r}_0}^{\bar{r}_1} (\sqrt{w_1(s)}+\mathcal{J}-B(r))ds\leq G_{w_2}(\bar{r}_0)+\int_{\bar{r}_0}^{\bar{r}_1} (\sqrt{w_2(s)}+\mathcal{J}-B(r))ds.
\end{equation*}
Since $w_1(r)> w_2(r)$ for $r\in(\bar{r}_0,\bar{r}_1)$, we get
\begin{equation*}
  G_{w_1}(\bar{r}_0)< G_{w_2}(\bar{r}_0),
\end{equation*}
which gives
$(w_1)_r(\bar{r}_0)< (w_2)_r(\bar{r}_0)$. This is a contradiction to \eqref{a2.11}. Therefore, the interior subsonic solution of \eqref{2.1} is unique.

In the end, we show that $m\in C^{\frac{1}{2}}\rol$. Since $m(r)\geq \mathcal{J}$ over $\rol$, then
\begin{equation*}
|m(a)-\mathcal{J}+m(c)-\mathcal{J}|=|m(a)-\mathcal{J}|+|m(c)-\mathcal{J}|\geq |(m(a)-\mathcal{J})-(m(c)-\mathcal{J})|=|m(a)-m(c)|.
\end{equation*}
Thus, by $(m-\mathcal{J})^2\in C^1\rol$, it is easy to see that
\begin{equation*}
\begin{split}
  \frac{|m(a)-m(c)|^2}{|a-c|}&=\frac{|m(a)-m(c)||(m(a)-\mathcal{J})^2-(m(c)
  -\mathcal{J})^2|}{|a-c||m(a)-\mathcal{J}+m(c)-\mathcal{J}|}\\
  &\leq
  \frac{|(m(a)-\mathcal{J})^2-(m(c)-\mathcal{J})^2|}{|a-c|}\\
  &\leq C,
  \end{split}
\end{equation*}
for any $a$, $c\in \rol$, which implies $m\in C^{\frac{1}{2}}\rol$.
This finishes the proof.
\end{proof}

{\it 2.2. the case of n=3.}
In the subsection, we prove the existence and uniqueness of interior subsonic solutions of \eqref{a1.4} in the three-dimensional case. Here \eqref{a1.4} is rewritten as
\be\label{2.9}
\begin{cases}
\left[r^{2}\left(\dfrac{1}{m}-\dfrac{\mathcal{J}^2}{m^3}\right)m_r+
\dfrac{r^{2}\mathcal{J}}{\T m}\right]_r=m-B(r)+2,\quad r\in\ro,\\
m(r_0)=m(r_1)=\mathcal{J}.
\end{cases}
\ee
Now we list some results for interior subsonic solution of \eqref{2.9}.
\begin{theorem}
Suppose that $\overline{\mathcal{B}}> \mathcal{J}$ and $\inf\limits_{r\in\rol}\left\{B(r)+\dfrac{2 r\mathcal{J}}{\T\overline{\mathcal{B}}}-2\right\}> \mathcal{J}$, then \eqref{2.9} admits a unique  interior subsonic solution $m(r)$ over $\rol$ satisfying $m\in C^{\frac{1}{2}}\rol$ and
\begin{equation*}
  m(r)\geq \mathcal{J}+\bar{\lambda}\sin \left(\pi\cdot\frac{r-r_0}{r_1-r_0}\right),\quad r\in\rol,
\end{equation*}
where $\bar{\lambda}$ is a small and positive constant.
\end{theorem}
\begin{proof}
First we divide the process into three steps.

{\it Step 1.} In this step, we concern  the following approximate equation of \eqref{2.9}
\be\label{a2.10}
\begin{cases}
\left[r^{2}\left(\dfrac{1}{m_j}-\dfrac{j^2}{(m_j)^3}\right)(m_j)_r+\dfrac{r^{2}\hj}{\T m_j}\right]_r=m_j-B(r)+2,\quad r\in\ro,\\
m(r_0)=m(r_1)=\mathcal{J},
\end{cases}
\ee
and prove the existence and uniqueness of the solution to \eqref{a2.10}.
In order to apply the Schauder fixed point theorem, we define an operator $\mathcal{P}:\bar{m}\rightarrow m_j$, by solving the linear equation
\bc
\begin{cases}
\left[r^2\left(\dfrac{1}{\bar{m}}-\dfrac{j^2}{\bar{m}^3}\right)(m_j)_r\right]_r-\dfrac{r^2 \hj}{\T \bar{m}^2}(m_j)_r=m_j-B(r)+2-\dfrac{2r \hj}{\T \bar{m}},\quad r\in\ro,\\
m_j(r_0)=m_j(r_1)=\mathcal{J}.
\end{cases}
\ec
Now it is easy to verify  that the fixed-point operator $\mathcal{P}$ is precompact and continuous. What's important is  to prove $\mathcal{P}(\mathcal{C})\subset\mathcal{C}$. As similar as that of Lemma \ref{l2.3}, and by applying the weak maximum principle, we get the result
\begin{equation*}
  \mathcal{J}\leq m_j(r)\leq \overline{\mathcal{B}}
\end{equation*}
provided that
\begin{equation*}
  \overline{\mathcal{B}}> \mathcal{J} \quad\text{and}\quad \inf\limits_{r\in\rol}\left\{B(r)+\dfrac{2r\hj}{\T\overline{\mathcal{{B}}}}-2\right\}> \mathcal{J}.
\end{equation*}
Hereafter
there exists  a fixed point $m_j$ of $\mathcal{P}$ such that $\mathcal{P}(m_j)=m_j$, which is also
a weak solution to \eqref{a2.10} satisfying $m_j\in H^2\ro$.

The uniqueness of the solution of \eqref{a2.10} can be obtained by a comparison principle, just like Lemma \ref{l2.1}.
Of course, we calculate that the comparison principle must be derived in the three-dimensional case. Hence, define
\begin{equation*}
\mathfrak{\bar{m}}(r):=\mathcal{J}+\bar{\lambda }\sin \left(\pi\cdot\frac{r-r_0}{r_1-r_0}\right),\quad r\in\rol,
\end{equation*}
and note that $\inf\limits_{r\in\rol}\left\{B(r)+\dfrac{2r\hj}{\T\overline{\mathcal{{B}}}}-2\right\}> \mathcal{J}$.
Then if $\bar{\lambda}>0$ is sufficiently small, we also obtain
\begin{equation*}\begin{split}
-\left[r^2\left(\dfrac{1}{\mathfrak{\bar{m}}}-\dfrac{j^2}{\mathfrak{\bar{m}}^3}\right)\mathfrak{\bar{m}}_r+\dfrac{r^2 \hj}{\T \mathfrak{\bar{m}}}\right]_r+\mathfrak{\bar{m}}-B(r)+2&\leq C(\bar{\lambda}^2+\bar{\lambda})+\left(\mathcal{J}-B(r)-\frac{2r\hj}{\T(\mathcal{J}+\bar{\lambda})}+2\right)\\
&< C(\bar{\lambda}^2+\bar{\lambda})+\left(\mathcal{J}-B(r)-\frac{2r\hj}{\T\overline{\mathcal{B}}}+2\right)\\
&<0.
\end{split}
\end{equation*}
Here $C$ is a positive constant independent of $j$.
By the comparison principle, we also get
\begin{equation*}
m_j(r)\geq \mathfrak{\bar{m}}(r)\quad \text{over}\quad \rol.
\end{equation*}

{\it Step 2.} The second step is to give a uniform bound estimate of the approximate solution $m_j(r)$ for all $0<j<\hj$. As in \eqref{2.5}, we have
\be\label{c2.7}
\begin{split}
   & (\mathcal{J}^2-j^2)\intr r^2\dfrac{|(m_j)_r|^2}{(m_j)^3}dr+\frac{4}{9}\intr r^2\frac{m_j+\mathcal{J}}{(m_j)^3}\cdot|[(m_j-\mathcal{J})^{\frac{3}{2}}]_r|^2dr \\
   &+\frac{\hj}{\T}\intr \frac{r^2(m_j)_r}{m_j}dr+\intr (m_j-B+2)(m_j-\mathcal{J})dr=0.
\end{split}
\ee
Then because of  $\underline{\mathcal{B}}>\hj$, it holds that
\begin{equation*}
\begin{split}
  \frac{\hj}{\T}\intr \frac{r^2(m_j)_r}{m_j}dr &=\frac{\hj}{\T}[(r_1^2-r_0^2)\cdot\ln \mathcal{J}]-\frac{\hj}{\T}\intr 2r\ln m_jdr,\\
  \end{split}
\end{equation*}
and
\begin{equation*}
\begin{split}
 \intr (m_j-B+2)(m_j-\mathcal{J})dr&\geq \frac{1}{2}\intr(m_j-\mathcal{J})^2dr-\frac{1}{2}\intr (B+\frac{2r}{\T}-2-\mathcal{J})^2dr\\&\quad+\frac{1}{\T}\intr 2r(m_j-\mathcal{J})dr,
\end{split}
\end{equation*}
where we used Young's inequality and integration by parts.
Therefore, it follows from \eqref{c2.7} and $\mathcal{J} \leq m_j \leq \overline{\mathcal{B}}$   that
\begin{equation*}
\begin{split}
&\frac{(\mathcal{J}^2-j^2)r_0^2 }{\overline{\mathcal{B}}^3}\intr |(m_j)_r|^2dr+\frac{8 r_0^2\mathcal{J}}{9\overline{\mathcal{B}}^3}\intr |[(m_j-\mathcal{J})^{\frac{3}{2}}]_r|^2dr+\frac{1}{2}\intr(m_j-J_0)^2dr\\
&\leq\frac{1}{2}\intr (\overline{\mathcal{B}}-\mathcal{J})^2dr+\frac{\mathcal{J}(r_1^2-r_0^2)(\ln \overline{\mathcal{B}}-\ln \mathcal{J})}{\T},
\end{split}
\end{equation*}
which also gives
\begin{equation*}
  ||(m_j-\mathcal{J})^{\frac{3}{2}}||_{H^1}\leq C\quad\text{and}\quad ||(\mathcal{J}^2-j^2)(m_j)_r||_{L^2}\leq C(\mathcal{J}^2-j^2)^{\frac{1}{2}},
\end{equation*}
for  some constant $C$ depending on $(\T,\overline{\mathcal{B}} ,r_0,r_1)$, but independent of $j$.
Hence, by the above estimates, there exists a subsequence $\{m_j\}_{0<j<\mathcal{J}}$, converging weakly to a limit $m$ as $j\rightarrow \mathcal{J}^-$. In fact, the limit function $m$ is certainly a weak solution of \eqref{2.9} such that $(m-\mathcal{J})^2\in H^1_0\ro$ and \eqref{1.14} holds.

{\it Step 3.} The last step is to prove the uniqueness of this interior subsonic solution $m(r)$ and to show $ m\in C^\frac{1}{2}\rol$. This part of the proof is referring to that of Theorem \ref{t2.1} directly,  and we don't repeat it here.
The proof is finished.
\end{proof}

\section{Existence of interior supersonic solutions}

In this section, we are going to  prove the existence of interior supersonic solutions of \eqref{a1.4} in the two and three dimensional cases, respectively.

{\it 3.1. the case of $n=2$.}\\
As similar as Lemma \ref{l2.1}, we introduce a comparison principle first.
\begin{lemma}\label{l3.1}
Let $V\in C^1\rol$ satisfying $V(r)\geq k_0>1$ over $\rol$ be a weak solution of the following equation
\begin{equation*}
\begin{cases}
\left[rd_1(r)\cdot(V-1) V_r+\dfrac{rV}{\T}\right]_r-\left(\dfrac{V}{\T}-d_2(r)\right)=0, \\
V(r_0)=V(r_1)=k_0,
\end{cases}r\in(r_0,r_1),
\end{equation*}
where $d_1,d_2\in L^{\infty}\ro$ and $d_1(r)>0$ on $\rol$. Thus, for any $\varphi\in H_0^1\ro$, it holds that
\begin{equation*}
  \intr \left[rd_1(r)\cdot(V-1) V_r+\dfrac{rV}{\T}\right]\varphi_rdr+\intr \left(\dfrac{V}{\T}-d_2(r)\right)\varphi dr=0.
\end{equation*}
In addition, let $U\in C^1\rol$ be such that $U(r)>0$ over $\rol$, $U(r_0)\leq k_0$, $U(r_1)\leq k_0$, and for any $\varphi\geq 0$, $\varphi\in H_0^1\ro$,
\begin{equation*}
  \intr \left[rd_1(r)\cdot(U-1) U_r+\dfrac{rU}{\T}\right]\varphi_rdr+\intr \left(\dfrac{U}{\T}-d_2(r)\right)\varphi dr\leq0.
\end{equation*}
Then $V(r)\geq U(r)$ over $\rol$.
\end{lemma}
\begin{proof}
Referring to the textbook \cite{Di01} (see Theorem 10.7) and Theorem 2.2 \cite{Li17}, we set
 \begin{equation*}
   I(r,z_1,z_2):=rd_1(r)(z_1-1)z_2+\frac{rz_1}{\T}.
 \end{equation*}
 Then, for any $\varphi\in H_0^1\ro$, $\varphi\geq 0$, we obtain
 \begin{equation}\label{a3.1b}
   \intr \left[ I(r,U,U_r)-I(r,V,V_r)\right]\varphi_rdr +\frac{1}{\T}\intr (U-V)\varphi dr\leq 0.
 \end{equation}
 Denote $W=:U-V$ and $U_t:=tU+(1-t)V$. A simple computation indicates that
 \begin{equation*}
 \begin{split}
   I(r,U,U_r)-I(r,V,V_r)&= I(r,U,U_r)-I(r,V,U_r)+I(r,V,U_r)-I(r,V,V_r) \\
     & =\int_0^1 \frac{\partial I}{\partial z_1}(r, U_t, U_r)dt\cdot W(r)+\int_0^1 \frac{\partial I}{\partial z_2}(r,V,(U_t)_r)dt\cdot W_r(r).
 \end{split}
 \end{equation*}
Let $\varphi(r)=\dfrac{W^+(r)}{W^+(r)+\epsilon}$ with $W^+(r):=\max\{0, W(r)\}$ and a positive constant $\epsilon$, and
note that
\begin{equation*}
\left[\ln \left(1+\frac{W^+(r)}{\epsilon}\right)\right]_r=\frac{W^+_r(r)}{W^+(r)+\epsilon},\quad
\varphi_r(r)=\frac{\epsilon}{W^+(r)+\epsilon}\left[\ln \left(1+\frac{W^+(r)}{\epsilon}\right)\right]_r.
\end{equation*}
Because $k_0>1$, $U\in C^1\rol$ and $ d_{m}:=\min\limits_{r\in\rol} d_1(r)>0$, this yields that
\begin{equation*}
  \int_0^1 \frac{\partial I}{\partial z_1}(r, U_t, U_r)dt=rd_1(r)U_r+\frac{r}{\T}\leq C
\end{equation*}
and
\begin{equation*}
\int_0^1 \frac{\partial I}{\partial z_2}(r,V,(U_t)_r)dt=rd_1(r)(V-1)\geq r_0(k_0-1)d_m.
\end{equation*}
Then it follows from \eqref{a3.1b} that
\begin{equation*}\begin{split}
  &\epsilon r_0(k_0-1)d_m\intr \left|\left[\ln \left(1+\frac{W^+(r)}{\epsilon}\right)\right]_r\right|^2dr
  +\frac{1}{\T} \intr \frac{(W^+(r))^2}{W^+(r)+\epsilon}dr\\
  &\leq C\epsilon \intr \frac{W^+(r)}{W^+(r)+\epsilon}\left|\left[\ln \left(1+\frac{W^+(r)}{\epsilon}\right)\right]_r\right|dr\\
   & \leq \frac{\epsilon r_0(k_0-1)d_m}{2}\intr \left|\left[\ln \left(1+\frac{W^+(r)}{\epsilon}\right)\right]_r\right|^2dr+\frac{C^2\epsilon(r_1-r_0)}
   {2r_0(k_0-1)d_m},
 \end{split}
\end{equation*}
where we used Young's inequality in the second inequality.
Thus, we get for any $\epsilon$
\begin{equation*}
\intr \left|\left[\ln \left(1+\frac{W^+(r)}{\epsilon}\right)\right]_r\right|^2dr\leq \frac{C^2(r_1-r_0)}{r_0^2(k_0-1)^2 d_m^2},
\end{equation*}
which further by Poincar\'{e}'s inequality gives
\begin{equation*}
\begin{split}
 \intr \left[\ln \left(1+\frac{W^+(r)}{\epsilon}\right)\right]^2dr
 &\leq(r_1-r_0)^2\intr\left|\left[\ln \left(1+\frac{W^+(r)}{\epsilon}\right)\right]_r\right|^2dr\\
    & \leq \frac{C^2(r_1-r_0)^3}{r_0^2(k_0-1)^2 d_m^2}<\infty.
\end{split}
\end{equation*}
Now letting $\epsilon \rightarrow 0^+$, one can see that if $W^+(r)\neq 0$ for some $r\in\rol$,
\begin{equation*}
 \intr \left[\ln \left(1+\frac{W^+(r)}{\epsilon}\right)\right]^2dr=\infty,
\end{equation*}
which
gets a contradiction. Therefore, $U(r)\leq V(r)$ for all $r\in\rol$.
\end{proof}

Then let's show the existence theorem as follows.
\begin{theorem}\label{t2.3}
Assume that $\underline{B}+\dfrac{1}{\T}> \mathcal{J}$, then system \eqref{2.1} admits at least one interior supersonic solution $m\in C^{1/2}\rol$ satisfying $\ell\leq m(r)\leq \mathcal{J}$ over $\rol$ for a positive constant $\ell$. Moreover, the function $m$ possesses the property that $m(r)<\mathcal{J}$ for any $r\in (r_0,r_1)$.
\end{theorem}
\begin{proof}
This proof is divided into three steps for clarity.

{\it Step 1.} We first consider the following approximate equation of \eqref{2.1}
\be\label{a1.6}
\begin{cases}
\left[r\left(\dfrac{1}{m_k}-\dfrac{k^2}{m_k^3}\right)(m_k)_r+\dfrac{r k}{\T m_k}\right]_r=m_k-B(r),\\
m_k(r_0)=m_k(r_1)=\mathcal{J},
\end{cases}
\ee
with the parameter $k> \mathcal{J}$.
Let $v_k(r):=\dfrac{k}{m_k(r)}$, thus, \eqref{a1.6} becomes
\begin{equation}\label{a1.7}
\begin{cases}
\left[r\left(v_k-\dfrac{1}{v_k}\right)(v_k)_r+\dfrac{rv_k}{\T}\right]_r-\left(\dfrac{k}{v_k}-B\right)=0, r\in(r_0,r_1),\\
v_k(r_0)=v_k(r_1)=\dfrac{k}{\mathcal{J}}\triangleq k_0>1.
\end{cases}
\end{equation}
Next we only need to prove that there exists a weak solution $v_k(r)$ to \eqref{a1.7} satisfying $v_k \geq k_0$. Here our adopted approach is the iterative method. Due to the effect of high dimensions space, we apply a so-called two-steps iteration to complete the following proof.

Let $X$ be a solution space, denoted by
\begin{equation*}\begin{split}
&X:=\{\phi(r):\phi\in C^1\rol, k_0\leq \phi(r)\leq \mathcal{M}, \phi(r_0)=\phi(r_1)=k_0,\\
&\quad \quad ||\phi||_{C^\alpha\rol}\leq \Lambda,||\phi||_{C^1\rol}\leq \Upsilon(\Lambda)\}.
\end{split}
\end{equation*}
Here some positive constants $\mathcal{M}$, $\Lambda$ and $\Upsilon(\Lambda)$ are determined later. Then we  define an operator $\Psi: \eta\longrightarrow v$ by solving the quasi-linear system
\begin{equation}\label{2.10}
\begin{cases}
\left[\dfrac{r(\eta+1)}{\eta}\cdot(v-1)v_r\right]_r+\dfrac{rv_r}{\T}-\left(\dfrac{k}
{\eta}-B-\dfrac{\eta}{\T}\right)=0, \quad r\in(r_0,r_1),\\
v(r_0)=v(r_1)=k_0,
\end{cases}
\end{equation}
where $\eta\in X$. To use the Schauder fixed point theorem,
we first claim that system \eqref{2.10} has a unique solution $v\in C^{1+\alpha}\rol$ for $0<\alpha<1$ and arbitrary fixed $\eta\in X $.

To this end, set $$S:=\{\omega\in C^0\rol| k_0\leq\omega\leq \mathcal{K}\quad\text{and}\quad\omega (r_0)=\omega(r_1)=k_0\}$$
 for a undetermined constant $\mathcal{K}$, and let's define a fixed-point operator $i:S\rightarrow C^0\rol$, $i(\xi)=\zeta$  by solving the linearized system of \eqref{2.10}
\begin{equation}\label{a3.5}
\begin{cases}
\left[rg_1(\eta,\xi)\cdot \zeta_r\right]_r+\dfrac{r\zeta_r}{\T}+g_2(\eta,\T)=0,\quad r\in(r_0,r_1), \\
\zeta(r_0)=\zeta(r_1)=k_0.
\end{cases}
\end{equation}
Here $\xi\in {S}$ and we have defined $g_1:=\dfrac{r(\eta+1)(\xi-1)}{\eta}$ and $g_2:=B+\dfrac{\eta}{\T}-\dfrac{k}{\eta}$. Furthermore, note that $g_1$ and $g_2$ are $C^1$-continuous with respect to $\eta$.

Actually, one finds that the fixed point of the operator $i$ is a solution of \eqref{2.10}, so we need to prove the existence of the fixed point of $i$.
Since \eqref{a3.5} has a solution $\zeta\in H^1\ro$, the operator $i$ is precompact by the compact imbedding $H^1\ro\hookrightarrow C^0\rol$. The continuity of $i$ is  based on the standard argument, obviously. Next, we only need to prove $k_0\leq \zeta(r)\leq \mathcal{K}$ over $\rol$.
Multiplying \eqref{a3.5} by $(\zeta-k_0)^-(r):=\min\left\{0,(\zeta-k_0)(r)\right\}$, we have
\begin{equation}\label{a3.5b}
\begin{split}
  &\intr rg_1(r)|[(\zeta-k_0)^-]_r|^2dr
  +\frac{1}{2\T}\intr [(\zeta-k_0)^-]^2dr+\intr (-g_2(r))
  (\zeta-k_0)^-dr=0
\end{split}
\end{equation}
where we have used
\bc
  \frac{1}{\T}\intr r\zeta_r(\zeta-k_0)^-dr=-\frac{1}{2\T}\intr [(\zeta-k_0)^-]^2dr.
\ec
Here each term of \eqref{a3.5b} is non-negative since $g_1\geq r_0(k_0-1)>0$, $\eta\geq k_0$, $B+\frac{k_0}{\T}> \hj$ and
\bc
 g_2(\eta,\T)=B+\dfrac{\eta}{\T}-\dfrac{k}{\eta}=\left(B+\frac{k_0}{\T}-\hj\right)+\left(\hj-\dfrac{k}{\eta}\right)
 +\dfrac{\eta- k_0}{\T}>0.
\ec
Then this implies by \eqref{a3.5b} that $\zeta(r)\geq k_0$ over $\rol$. Now multiplying  \eqref{a3.5} by $(\zeta-k_0)(r)$, we show that
\begin{equation*}
\begin{split}
  &r_0(k_0-1)\intr |(\zeta-k_0)_r|^2dr+\frac{1}{2\T}\intr (\zeta-k_0)^2dr\leq\intr g_2(\eta,\T)(\zeta-k_0)dr.\\
\end{split}
\end{equation*}
Thus, by Young's inequality and Poincar\'{e}'s inequality
\begin{equation*}
  \intr (\zeta-k_0)^2dr
\leq (r_1-r_0)^2\int_{r_0}^{r_1}|(\zeta-k_0)_r|^2dr,
\end{equation*}
we get
\begin{equation*}
\begin{split}
  r_0(k_0-1)\intr |(\zeta-k_0)_r|^2dr&\leq\intr g_2(\eta,\T)(\zeta-k_0)dr\\
  &\leq \frac{r_0(k_0-1)}{2(r_1-r_0)^2}\intr (\zeta-k_0)^2dr+\frac{(r_1-r_0)^2}{2r_0(k_0-1)}\intr g_2^2(\eta,\T)dr\\
  &\leq \frac{r_0(k_0-1)}{2}\intr |(\zeta-k_0)_r|^2dr+\frac{(r_1-r_0)^2}{2r_0(k_0-1)}\|g_2(\eta)\|_{L^2}^2,
\end{split}
\end{equation*}
which indicates that
\begin{equation*}
  \|(\zeta-k_0)_r\|_{L^2}\leq \frac{r_1-r_0}{r_0(k_0-1)}\|g_2(\eta)\|_{L^2}.
\end{equation*}
Further, we conclude that
\begin{equation*}
  \zeta\leq k_0+C(r_0,r_1,k_0)\|g_2(\eta)\|_{L^2},
\end{equation*}
then choose $\mathcal{K}(\eta):=k_0+C(r_0,r_1,k_0)\|g_2(\eta)\|_{L^2}$ such that $k_0\leq \zeta\leq \mathcal{K}$. Applying the Schauder fixed point theorem, we have a fixed point $v\in S$ such that $i(v)=v$, which is also  a weak solution to \eqref{2.10}. Thanks to the regularity theory and Sobolev imbedding theory \cite{Di01}, it's proved that $v\in C^{1+\alpha_0}\rol$ such that $k_0\leq v\leq \mathcal{K}(\eta)$,
\begin{equation}\label{a3.9}
\|v\|_{C^{\alpha_0}\rol}\leq C_0(k_0,\eta,\T,\mathcal{K}(\eta))\quad\text{and}\quad\|v\|_{C^{1+\alpha_0}\rol}\leq C(C_0,k_0,\eta,\T,\mathcal{K}(\eta))
\end{equation}
for constants $C_0$ and $0<\alpha_0<\frac{1}{2}$. Moreover, we can prove that  the solution of \eqref{2.10} is unique by Lemma \ref{l3.1}. Thus the claim is verified.

Next, we go back to show that $\Psi$ has a fixed point, so it is necessary to prove $\Psi(X)\subset X$. Since $v(r)\geq k_0$ over $\rol$, it remains to determine the upper bound of the solution $v(r)$ for \eqref{2.10}.
Multiplying  \eqref{2.10} by $(v-k_0)^2$, we derive
\bc
\begin{split}
  &\intr \dfrac{r(\eta+1)}{2\eta}\cdot|[(v-k_0)^2]_r|^2dr+\frac{1}{3\T}\intr (v-k_0)^3dr\leq\intr\left(B+\dfrac{\eta}{\T}-\dfrac{k}{\eta}\right)(v-k_0)^2dr,\\
\end{split}
\ec
which leads to
 \begin{equation*}\begin{split}
  \frac{r_0}{2}\intr |[(v-k_0)^2]_r|^2dr
  &\leq\intr\left({B}+ \dfrac{\eta}{\T}\right)(v-k_0)^2dr\\
  &\leq\frac{r_0}{4(r_1-r_0)^2}\intr(v-k_0)^4dr+\frac{(r_1-r_0)^2}{r_0}\intr\left({B}+ \dfrac{\eta}{\T}\right)^2dr\\
  &\leq\frac{r_0}{4}\intr |[(v-k_0)^2]_r|^2dr+\frac{(r_1-r_0)^3}{r_0}\left(\overline{B}+ \dfrac{\mathcal{M}}{\T}\right)^2.
\end{split}
\end{equation*}
Here we have used Poincar\'{e}'s inequality
\begin{equation*}
  \intr (v-k_0)^4dr\leq (r_1-r_0)^2\int_{r_0}^{r_1}|[(v-k_0)^2]_r|^2dr.
\end{equation*}
It then follows that
\bc
\|[(v-k_0)^2]_r\|_{L^2}^2\leq \frac{4(r_1-r_0)^3}{ r_0^2}\left(\overline{B}+ \dfrac{\mathcal{M}}{\T}\right)^2.
\ec
Moreover it holds that
\begin{equation*}
  0<v(r)\leq k_0+C\sqrt{\overline{B}+ \dfrac{\mathcal{M}}{\T}}
\end{equation*}
for a positive constant $C$ depending on $(r_0,r_1)$. Thus by a simple calculation, we can choose $$\mathcal{M}=\mathcal{M}(\overline{B},\T)\geq k_0+\dfrac{C^2}{2\T}+C\sqrt{\overline{B}+\dfrac{k_0}{\T}+\dfrac{C^2}{4\T^2}}$$ such that $k_0+C\sqrt{\overline{B}+ \dfrac{\mathcal{M}}{\T}}\leq \mathcal{M}. $  Then we can see that $v(r)\leq \mathcal{M}$ over $\rol$ for any $k_0\leq\eta\leq \mathcal{M}$.
Hereafter it implies by \eqref{a3.9} that for a constant $0<\alpha_0<\frac{1}{2}$,
\begin{equation*}
  \|v\|_{C^{\alpha_0}\rol}\leq C_0(\mathcal{M},\mathcal{K}(\mathcal{M}),\T,k_0)\quad \text{ and}\quad \|v\|_{C^{1+\alpha}\rol}\leq C(\mathcal{M},\mathcal{K}(\mathcal{M}),\T,k_0,C_0),
\end{equation*}
and we determine $\alpha=\alpha_0$, $\Lambda=C_0(\mathcal{M},\mathcal{K}(\mathcal{M}),\T,k_0)$ and $\Upsilon(\Lambda)=C(\mathcal{M},\mathcal{K}(\mathcal{M}),\T,k_0,\Lambda)$.  Now it can be verified that $v\in X$ and $X$ is a bounded and closed convex subset of $C^1\rol$. Also, the operator $\Psi$ is a compact map of $X$ into itself by the compact imbedding $C^{1+\alpha}\rol\hookrightarrow C^{1}\rol$. Using the continuity theory, one can see that  the operator $\Psi$ is continuous. Hence, a fixed point of the map $\Psi$ can be obtained by the Schauder fixed point theorem. In the end, \eqref{a1.7} has a weak solution $v_k\in C^1\rol$, and $m_k(r)=k/v_k(r)$ is an interior supersonic solution of \eqref{a1.6} over $\rol$.

{\it Step 2.} This step is to prove the existence of the interior supersonic solutions of \eqref{2.1}. Multiplying \eqref{a1.7} by $(v_k-k_0)(r)$, and using Young's inequality and Poincar\'{e}'s inequality, we have
\begin{equation*}\begin{split}
  &(k_0-1)\intr\dfrac{r(v_k+1)}{v_k}|(v_k)_r|^2dr+\frac{4}{9}\intr \dfrac{r(v_k+1)}{v_k}|[(v_k-k_0)^\frac{3}{2}]_r|^2dr\\&=
  \intr\left(B+\dfrac{v_k+k_0}{2\T}-\dfrac{k}{v_k}\right)(v_k-k_0)dr\\
  &\leq  \frac{2}{3}\intr (v_k-k_0)^\frac{3}{2}dr+  \frac{1}{3}\intr \left(B+\dfrac{v_k+k_0}{2\T}\right)^3dr\\
  &\leq\frac{r_0}{3(r_1-r_0)^2}\intr (v_k-k_0)^3dr+\frac{(r_1-r_0)^3}{r_0}+\frac{r_1-r_0}{3} \left(\overline{B}+\dfrac{\mathcal{M}+k_0}{2\T}\right)^3\\
  &\leq \frac{r_0}{3}\intr |[(v_k-k_0)^\frac{3}{2}]_r|^2dr+C(\overline{B},\T,\mathcal{M},k_0,r_0,r_1),
\end{split}
\end{equation*}
where we used
\begin{equation*}
  \intr \frac{rv_k}{\T}(v_k-k_0)_rdr=-\dfrac{1}{2\T}\intr (v_k+k_0)(v_k-k_0)dr.
\end{equation*}
Thus, it follows that
\begin{equation*}
||(k_0-1)^\frac{1}{2}(v_k)_r||_{L^2}+||(v_k-k_0)^\frac{3}{2}||_{H^1}\leq C,
\end{equation*}
for some constants $C$ only depending on $(\overline{B},\T, \mathcal{M},k_0,r_0,r_1)$.
In fact, as $k_0\rightarrow1^+$, i.e. $k\rightarrow \hj^+$, given that a suitable choice of $\mathcal{M}$,
 we can obtain
\begin{equation*}
  \|(v_k-k_0)^\frac{3}{2}\|_{L^{\infty}}\leq C(\overline{B},\T,r_0,r_1),
\end{equation*}
which gives
\begin{equation*}
  v_k\leq k_0+C^{\frac{2}{3}}.
\end{equation*}
Then,
\begin{equation}\label{2.13}
  m_k(r)=\frac{k}{v_k(r)}\geq \frac{k}{k+C^{\frac{2}{3}}}\geq\frac{1}{1+C^{\frac{2}{3}}}\triangleq \ell\quad\text{for}\quad r\in\rol.
\end{equation}
A direct computation yields that
\begin{equation*}
  (m_k)_r=-\frac{k(v_k)_r}{v_k^2}\quad\text{and}\quad ((\mathcal{J}-m_k)^2)_r=\frac{4\mathcal{J}k(v_k-k_0)^\frac{1}{2}((v_k-k_0)^\frac{3}{2})_r}{3v_k^3},
\end{equation*}
which together with \eqref{2.13} implies
\begin{equation*}
  ||(k_0-1)^\frac{1}{2}(m_k)_r||_{L^2}+\|(\mathcal{J}-m_k)^2\|_{H^1}\leq  C(\overline{B},\T,r_0,r_1).
\end{equation*}
Finally, one can see that there exists a function $m$ such that, as $k\rightarrow \mathcal{J}^+$, up to a subsequence,
\begin{equation}\label{a3.8}
\begin{split}
  & (\mathcal{J}-m_k)^2\rightharpoonup (\mathcal{J}-m)^2\quad \text{weakly~in}\quad H^1\ro, \\
  &(\mathcal{J}-m_k)^\frac{3}{2}\rightharpoonup (\mathcal{J}-m)^\frac{3}{2}\quad \text{weakly~in}\quad H^1\ro,\\
   &(\mathcal{J}-m_k)^\frac{3}{2}\rightarrow(\mathcal{J}-m)^\frac{3}{2}\quad \text{strongly~in}\quad C^\alpha\rol,\quad 0<\alpha<\frac{1}{2},\\
    & (k_0-1)(m_k)_r\rightarrow0\quad \text{strongly~in}\quad L^2\ro.
\end{split}
\end{equation}
Hence equation \eqref{2.1} has an interior supersonic solution $m(r)$ over $\rol$, and \eqref{1.14} holds. The lower bound of the solution $m$ is obtained by \eqref{2.13} and \eqref{a3.8}, and $m\in C^{1/2}\rol$ is easily obtained as similar as that of Theorem \ref{t2.1}.

{\it Step 3.} At the last step,  we need to prove that $m(r)<\mathcal{J}$ over $\ro$. If a function $m$ satisfies $m(r)\equiv \mathcal{J}$ on any interval $[s_1, s_2]\subset\rol$, then $m$ is not a solution of \eqref{2.1} because $\underline{B}+\frac{1}{\T}> \mathcal{J}$. Thus, there exist two points $\hat{s}_1$ and $\hat{s}_2$ satisfying
$0<\hat{s}_1-r_0\ll1$ and $0<r_1-\hat{s}_2\ll1$. Then let $\varepsilon>0$ be a small number such that $m(\hat{s}_1)$, $m(\hat{s}_2)\leq \mathcal{J}-\varepsilon<\hj$. Next, we are going to prove that $m(r)\leq \mathcal{J}-\varepsilon$ over $[\hat{s}_1, \hat{s}_2]$. After that, set $w=(\mathcal{J}-m)^2$, further we know $w\in H_0^1\ro$ and $w(\hat{s}_1)$, $w(\hat{s}_2)\geq \varepsilon^2$. From \eqref{1.14}, taking $\varphi(r)=(w-\varepsilon^2)^-(r)$, we have
\begin{equation}\label{3.7}
\begin{split}
&\int_{\hat{s}_1}^{\hat{s}_2} r\frac{2\mathcal{J}-\sqrt{w}}{2(\mathcal{J}-\sqrt{w})^3}|[(w-\varepsilon^2)^-]_r|^2dr
+\int_{\hat{s}_1}^{\hat{s}_2} \frac{r\mathcal{J}[(w-\varepsilon^2)^-]_r}{\T (\mathcal{J}-\sqrt{w})}dr\\
&\quad+\int_{\hat{s}_1}^{\hat{s}_2} (\mathcal{J}-\sqrt{w}-B(r))(w-\varepsilon^2)^-dr=0.
\end{split}
\end{equation}
Since $2\mathcal{J}-\sqrt{w}> \mathcal{J}-\sqrt{w}\geq0$, one can see that the first term of  \eqref{3.7} is non-negative. Then, by a direct computation, we change the last two of \eqref{3.7} as
\begin{equation}\label{3.8}
\begin{split}
&\int_{\hat{s}_1}^{\hat{s}_2} \frac{r\mathcal{J}\left[(w-\varepsilon^2)^-\right]_r}{\T (\mathcal{J}-\sqrt{w})}dr+\int_{\hat{s}_1}^{\hat{s}_2} (\mathcal{J}-\sqrt{w}-B(r))(w-\varepsilon^2)^-dr\\
&\quad=\dfrac{1}{\T}\int_{\hat{s}_1}^{\hat{s}_2}\frac{r\sqrt{w}\left[(w-
\varepsilon^2)^-\right]_r}{\mathcal{J}-\sqrt{w}}dr
+\dfrac{1}{\T}\int_{\hat{s}_1}^{\hat{s}_2} {r[(w-\varepsilon^2)^-]_r}dr\\
&\quad\quad+\int_{\hat{s}_1}^{\hat{s}_2} (\mathcal{J}-\sqrt{w}-B(r))(w-\varepsilon^2)^-dr\\
&\quad =\dfrac{1}{\T}\int_{\hat{s}_1}^{\hat{s}_2}\frac{r\sqrt{w}\left[(w-\varepsilon^2)^-
\right]_r}{\mathcal{J}-\sqrt{w}}dr+\int_{\hat{s}_1}^{\hat{s}_2} \left(\mathcal{J}-\sqrt{w}-B(r)-\frac{1}{\T}\right)(w-\varepsilon^2)^-dr.\\
\end{split}
\end{equation}
Here we notice that the second term on the right-hand side of \eqref{3.8} is also non-negative because of $\underline{B}+\frac{1}{\T}> \mathcal{J}$. It remains  to show the non-negativity of the first term on the right-hand side of \eqref{3.8}.
Note that
\begin{equation*}
-[h(\sqrt{w})]_r:=-\left[2\mathcal{J}\sqrt{w}+w+2\mathcal{J}^2 \ln(\mathcal{J}-\sqrt{w})\right]_r=\frac{\sqrt{w}w_r}{\mathcal{J}-\sqrt{w}},
\end{equation*}
then it follows that
\begin{equation}\label{3.9}
  \begin{split}
\dfrac{1}{\T}\int_{\hat{s}_1}^{\hat{s}_2}\frac{r\sqrt{w}\left[(w-\varepsilon^2)^-
\right]_r}{\mathcal{J}-\sqrt{w}}dr&
=-\dfrac{1}{\T}\int_{\hat{s}_1}^{\hat{s}_2} r[h(\sqrt{w^\varepsilon})]_rdr=\dfrac{1}{\T}\int_{\hat{s}_1}^{\hat{s}_2} [h(\sqrt{w^\varepsilon})-h(\varepsilon)]dr,
\end{split}
\end{equation}
where $w^\varepsilon:=(w-\varepsilon^2)^-+\varepsilon^2$. Hence,  $0\leq\sqrt{w^\varepsilon}\leq \varepsilon$. Then a simple computation yields that
\begin{equation*}
  h'(s)=2\mathcal{J}+2s-\frac{2\mathcal{J}^2}{\mathcal{J}-s}<0 \quad \text{for}\quad s\in (0,\varepsilon],
\end{equation*}
 because $h''(s)<0$ on $(0,\varepsilon)$ and $h'(0)=0$. Thus, it holds that the right side of \eqref{3.9} is non-negative, which leads to $(w-\varepsilon^2)^-=0$. We derive that $m(r)\leq\mathcal{J}-\varepsilon$ over $[\hat{s}_1, \hat{s}_2]$ for some small constants $\varepsilon$.
The proof is finished.
\end{proof}

{\it 3.2. the case of $n=3$.}
In the subsection, we state the results of interior supersonic solutions to \eqref{a1.4} in three dimensional case.

\begin{theorem}
Assume that $\underline{\mathcal{B}}> \mathcal{J}$, then system \eqref{2.9} admits an interior supersonic solution $m\in C^{1/2}\rol$ satisfying $\bar{\ell}\leq m(r)\leq \mathcal{J}$ over $\rol$ for a positive constant $\bar{\ell}$, moreover, $0<m(r)<\mathcal{J}$ over $\ro$.
\end{theorem}
\begin{proof} This proof is similar as that of Theorem \ref{t2.3}, so we sketch it as follows.
The approximate system of \eqref{2.9} is the following equation
\be\label{3.2}
\begin{cases}
\left[r^2\left(\dfrac{1}{m_k}-\dfrac{k^2}{m_k^3}\right)(m_k)_r+\dfrac{r^2 k}{\T m_k}\right]_r=m_k-B(r)+2,\\
m_k(r_0)=m_k(r_1)=\mathcal{J},
\end{cases}
\ee
with the parameter $k> \mathcal{J}$.
Let $v_k(r):=\dfrac{k}{m_k(r)}$, thus \eqref{3.2} can be recast as
\begin{equation}\label{3.3}
\begin{cases}
\left[r^2\left(v_k-\dfrac{1}{v_k}\right)(v_k)_r+\dfrac{r^2v_k}{\T}\right]_r-\left(\dfrac{k}{v_k}-B+2\right)=0, r\in(r_0,r_1),\\
v_k(0)=v_k(1)=k_0.
\end{cases}
\end{equation}
Then we define an operator $\tilde{\Psi}: \eta\longrightarrow v$ by solving the following system
\begin{equation}\label{3.4}
\begin{cases}
\left[\dfrac{r^2(\eta+1)}{\eta}\cdot(v-1)v_r\right]_r+\dfrac{r^2v_r}{\T}
-\left(\dfrac{k}{\eta}-B+2-\dfrac{2r\eta}{\T}\right)=0,\quad r\in(r_0,r_1), \\
v(0)=v(1)=k_0,
\end{cases}
\end{equation}
where $\eta\in X$. As similar to that of Theorem \ref{t2.3}, and by applying the Schauder fixed point theorem,  we show that there exists a unique solution $v\in C^{1+\alpha}\rol$ to the quasi-linear system \eqref{3.4}, and there exists a constant $\tilde{\mathcal{K}}$ depending on $\eta$ such that $k_0\leq v\leq\tilde{\mathcal{K}}$,
\begin{equation}\label{c3.16}
\|v\|_{C^{\alpha}\rol}\leq C_0(k_0,\eta,\T,\tilde{\mathcal{K}}(\eta))\quad\text{and}\quad\|v\|_{C^{1+\alpha}\rol}\leq C(C_0,k_0,\eta,\T,\tilde{\mathcal{K}}(\eta))
\end{equation}
for constants $\tilde{\mathcal{K}}(\eta)$ and $0<\alpha<\frac{1}{2}$.
In the following, we only need to prove $k_0\leq v_k\leq \mathcal{M}$ with a proper constant $\mathcal{M}$. Obviously, since $v\geq  k_0$, then we only prove $ v_k\leq \mathcal{M}$.
Now multiplying  \eqref{3.4} by $(v-k_0)^2$, we derive
\bc
\begin{split}
  &\intr \dfrac{r^2(\eta+1)}{2\eta}\cdot|[(v-k_0)^2]_r|^2dr+\frac{2}{3\T}\intr r(v-k_0)^3dr\\
  &\leq\intr\left(B+\dfrac{2r\eta}{\T}-2-\dfrac{k}{\eta}\right)(v-k_0)^2dr,\\
\end{split}
\ec
which follows from the proof in Theorem \ref{t2.3} that
\bc
\|[(v-k_0)^2]_r\|_{L^2}\leq C(r_0,r_1,k_0)\left(\overline{\mathcal{B}}+{\mathcal{M}}\right).
\ec
Thus, we get
\begin{equation*}
  0<v(r)\leq k_0+C(r_0,r_1,k_0)\sqrt{\left(\overline{\mathcal{B}}+\mathcal{M}\right)}.
\end{equation*}
Then take $\mathcal{M}=\mathcal{M}(\overline{\mathcal{B}},k_0)$ sufficiently large such that $$k_0+C(r_0,r_1,k_0)\sqrt{\left(\overline{\mathcal{B}}+\mathcal{M}\right)}\leq \mathcal{M}.$$  As a result, it holds that $v(r)\leq \mathcal{M}$ over $\rol$. Next it follows from \eqref{c3.16} that $v\in X$ and $X$ is also a bounded and closed convex subset of $C^1\rol$.
Hereafter the Sobolev imbedding theorem and the Schauder fixed point theorem yield that there exists a fixed  point $v_k$ of the operator $\tilde{\Psi}$ such that
\begin{equation*}
  \tilde{\Psi}(v_k)=v_k.
\end{equation*}
Hence equation \eqref{3.3} has a weak solution $v_k$, then a solution of \eqref{3.2} would be obtained, that is
\begin{equation*}
  m_k(r)=\frac{k}{v_k(r)},\quad r\in \rol.
\end{equation*}

 Next, the bound estimate of $m_k$ can  be verified by
 \begin{equation*}
  ||(k_0-1)^\frac{1}{2}(m_k)_r||_{L^2}+\|(\mathcal{J}-m_k)^2\|_{H^1}\leq  C(\overline{\mathcal{B}},k_0),
\end{equation*}
whose proof is shown in Theorem \ref{t2.3}. Moreover, as similar to that of Theorem \ref{t2.3}, there exists a limit of convergence $m(r)$ by  a subsequence $\{m_k\}_{\hj<k<+\infty}$ as $k\rightarrow\hj^+$, which is an interior supersonic solution of \eqref{2.9}. For a constant $0<\bar{\ell}<\hj$, it is easy to check that $\bar{\ell}\leq m(r)<\mathcal{J}$ over $\ro$ and $m\in C^{1/2}\rol$. The proof is complete.
\end{proof}

\section*{Acknowledgements}

The research of M. Mei  was partially supported by NSERC grant RGPIN 354724-2016 and FRQNT grant 2019-CO-256440. The research of G. Zhang was partially supported by NSF of China (No. 11871012). The research of K. Zhang was partially supported by NSF of China (No. 11771071).

\end{document}